\documentclass[10pt,reqno]{article}
\usepackage{amsfonts,amssymb,amsmath,amsthm,latexsym,graphicx,xcolor,euscript}
\usepackage{pgfplots} 
\usepackage{bbm}
\usepackage{hyperref}
\usepackage{subfigure, enumerate}
\usepackage{tikz}
\usepackage{ccaption}
\usepackage{url}
\usepackage{mathrsfs}
\usepackage{dsfont}
\usepackage{authblk}  
\usepackage{comment}

\usepackage{ushort}

\newcommand{\X}{\mathcal X} 

\newcommand{\E}{\mathds{E}}
\newcommand{\un}{\mathds{1}}
\newcommand{\N}{\ensuremath{\mathbb{N}}}

\newtheorem{remark}{Remark}

\newtheorem{lemma}{Lemma}

\newtheorem{defi}{Definition}[section]
\newtheorem{theo}{Theorem}

\newtheorem{prop}{Proposition}

\begin{document}



\title{Potential well in  Poincar\'e recurrence}
\author{Miguel Abadi, Vitor Amorim, Sandro Gallo}


\maketitle

\abstract{From a physical/dynamical system perspective, the potential well  represents the proportional mass of points that escape the neighbourhood
of a given point. In the last 20 years, several works have  shown the importance of this quantity to obtain precise approximations for several recurrence time  distributions in mixing stochastic processes and  dynamical systems. Besides providing a review of the different scaling factors used in the literature in recurrence times,  the present work contribute with two new results: 
(1) for $\phi$-mixing and $\psi$-mixing  processes, we  give a new exponential approximation  for hitting and return times using the  potential well as scaling parameter. The error terms are explicit and sharp.
(2) We analyse the uniform positivity of the potential well. }

\tableofcontents

\section{Introduction}
The close relation between the Extreme Value Theory (EVT) and the statistical properties of Poincar\'e recurrence has been recently quite well explored. The starting point is that the exceedances of a stochastic process to  a sequence  of barrier  values $a_n>0, n\in\mathbb N,$  can be considered as hittings to a sequence of nested sets. More precisely, if one defines the semi-infinite intervals
\[
A_n=(a_n,\infty)  ,
\]
and consider a sequence of random variables $X_1,X_2,\ldots$, one has the equivalence
\[
\max\{X_1,\dots,X_t\} > a_n \  \ \ \  \text{if and only if } \  \ \ \  T_{A_n} \le t ,
\] 
where for any measurable set $A$, $T_A$ denotes the smallest $k$ such that $X_k\in A$. As the sequence of levels $a_n$ diverges, the sets $A_n$ are nested. This equivalence allows to make a bridge between two historically independent theories: Extreme Value Theory (EVT) and Poincar\'e Recurrence Theory (PRT). 
While EVT focuses on the existence (and identification) of the limit of the distribution of the partial maxima,  $k$-maxima, among others (\cite{leadbetter2012extremes,freitas2010hitting,freitas2013extremal, resnick2013extreme,lucarini2016extremes}), the aim of recent works on PRT is to understand the statistical properties of the different notions of return times. 

\vspace{0.2cm}

The present paper stands in the approach of PRT, our  interest is about the statistics of visits of  a random process $X_t,t\in \mathbb N$ to a given target measurable set.
\emph{Asymptotic} statistics are obtained studying  sequences of target sets $A_n,n\ge1$, usually of measure shrinking to  zero. In this context and  for certain classes of processes, hitting and return times with respect to a given sequence of target sets converge to the exponential distribution, modelling the unpredictability of rare events. However, this rough affirmation is full of nuances which need to be established in very precise terms. It turns out that these details bring many information on the system. 

For instance, for two observables having the same probability, the Ergodic Theorem says that, macroscopically, their  number of occurrences are about the same. However, these occurrences can appear scattered in a very different way along time. Under some strong mixing assumptions, it is a well-known fact of the literature that for nested sequences of observables with the same probability, the asymptotic observation of one of them can be distributed as a Poisson process while the other one will follow a compound Poisson process. Thus, the dichotomy Poisson/compound Poisson in the \emph{same} system is determined also by the intrinsic properties of the target sets considered. 

In the setting of the present paper, the target sets are finite strings of symbols (\emph{patterns}). In this case, even if the process is a sequence of independent random variables, the successive occurrences of the string are not independent because the structure of the pattern itself enters the game, allowing or avoiding consecutive observations due to possible overlaps with itself. This leads to a dichotomy between aperiodic/periodic patterns which yields, in the limit of long patterns, to the dichotomy Poisson/compound Poisson mentioned before. In passing, let us also mention that this dichotomy also exists in EVT where it is referred as phenomenon of clustering/non clustering of maxima, and has generated a great deal of research along the 2 last decades \cite{lucarini2016extremes}.
 
 \vspace{0.2cm}
 
 Let us now be more specific about what we are doing here. First, we stand in the context of discrete time stochastic processes with countable alphabet enjoying $\phi$-mixing.  Fix any point $x$, that is, any right infinite sequence of symbols taken from the alphabet, and  consider the nested sequence of neighbourhoods corresponding to the first $n$ symbols of $x$, namely $A_n=(x_0,\ldots,x_{n-1}), n\ge1$. The main theorem of the paper, Theorem \ref{theo:correct}, gives explicit and computable error terms for the approximation of the hitting time distribution $\mu(T_{A_n}>t)$ and return time distribution $\mu_{A_n}(T_{A_n}>t)$, by  exponential distributions whose parameter is explicit and depends on $A_n$.

The first main advantage of Theorem \ref{theo:correct} is that it uses the \emph{potential well} as scaling parameter. In words, the potential well is the  probability, conditioned on starting from $A_n$, that the pattern $A_n$ does not reappear at the first possible moment it could reappear. The use of this simple and well defined quantity as scaling parameter  contrasts with previous works using  parameters whose expressions are hardly explicit and even more hardly computable. 
Another advantage of Theorem \ref{theo:correct} is that, unlike a whole body of literature obtaining almost sure results, our results hold for \emph{all $x$}. This allows to distinguish different limiting distributions, as for example in the periodic/aperiodic dichotomy described above that almost-sure results cannot detect.
And last but not least, the error terms of  our approximations are not in total variation distance, but in the stronger point-wise form with respect to the time scale. 

Yet another important point of Theorem \ref{theo:correct}, with respect to return times, is that it corrects the exponential approximation  obtained by \cite{abadi/vergne/2009}. Indeed, Theorem 4.1 therein contains a mistake in the error term for small $t$'s.

 \vspace{0.2cm}

The other main novelty of the present work is Theorem \ref{theo:positive}, stating that the potential well is uniformly bounded away from $0$ when we have $\psi$-mixing or $\phi$-mixing with summable function $\phi(n)$. Naturally, as a conditional probability, we know that the potential well belongs to the interval $[0,1]$ for any $n\ge1$ and any pattern $A_n$. But it was proved that the potential well could be arbitrarily close to $0$ for $\beta$-mixing processes, a slightly weaker mixing assumption than $\phi$-mixing. Indeed, it was shown  in \cite{abadi/cardeno/gallo/2015} that for the binary renewal process, with  specific choices of transition probabilities and target sets $A_n,n\ge1$,  the potential well of $A_n$ vanishes as $n$ diverges. Note that the border is thin between this $\beta$-mixing example  and our Theorem \ref{theo:positive} holding for $\psi$-mixing and $\phi$-mixing with summable $\phi$ (see the review of \cite{bradleysurvey} on the distinct mixing assumptions). We conjecture that the assumption of summability of the $\phi$ rates can be dropped. 

 \vspace{0.2cm}

To conclude on the importance of the present work as a whole, let us mention that our  results  are fundamental for the study  of further recurrence quantities, such as the return time function \cite{wyner1989some,ornstein/weiss/1993} and the waiting time function \cite{shields1993waiting, marton1994almost}, establishing a link with information theory. These random variables are known to satisfy a counterpart of the famous Shannon-McMillan-Breiman Theorem (asymptotic equipartition property). In order to study the fluctuations of these limit theorems, for instance a large deviation principle, we need to control the return/hitting time exponential approximations for any point and any $t>0$. It is particularly clear in \cite{chazottes/ugalde/2005} and \cite{abadi2019complete} studying the fluctuations of the waiting time and return time respectively. It is also interesting to notice that it was \cite{collet/galves/schmitt/1999} who first pointed the importance of seeking exponential approximations for \emph{any} point $x$, and it was precisely to  study the small and large fluctuations of the return time function. 

%
%

 \vspace{0.2cm}

The paper is organized as follows. We describe in Section \ref{sec:setting} the setting of the paper in the context of PRT, defining carefully the types of exponential approximations we are interested in and explaining, including through an extensive bibliography, the role of the potential well as scaling parameter. Section \ref{sec:results} contains the main results and Section \ref{sec:proofs} is dedicated to their proofs. 

\section{Poincaré Recurrence Theory for mixing processes}\label{sec:setting}

\subsection{The framework of mixing processes}


Consider a countable set $\mathcal{A}$ that we call  alphabet.
With $\N$ we denote the set of nonnegative integers and with  $\mathcal X:=\mathcal{A}^{\mathbb{N}}$ the set of right infinite sequences  $x=(x_0,x_1,\ldots)$ of symbols taken from $\mathcal{A}$.
 Given a point $x\in \mathcal{X}$, and for any finite set $I\subset\N$, the cylinder sets with base in $I$
is defined as  the set $A_I(x):=\{y\in \mathcal X :  y_i=x_i,i\in I\}$. In the particular case where $I=\{0\ldots,n-1\}$ we will write  $A_n(x)$ and sometimes abuse notation writing $x_0^{n-1}$. We endow $\mathcal X$ with the  $\sigma$-algebra $\mathcal F$  generated by the class of cylinder sets 
$\{A_I:I\subset\N,|I|<\infty\}$. 
Further  $\mathcal F_I$ denotes  the $\sigma$-algebra generated by $A_I(x), x\in \mathcal{X}$. 
The special case in which $I=\{i,\ldots,j\}$, $0\le i\le j\le \infty$, we use the notation $\mathcal F_i^j$. We  use the shorthand notation $a_i^j:=(a_i,a_{i+1},\ldots,a_j)$, $0\le i\le j<\infty$ for finite strings of consecutive symbols of $\mathcal{A}$. 
When necessary,  $A_n(x)$ will be naturally identify with the sequence $x_0^{n-1}$.

The shift operator $\sigma:\mathcal X\rightarrow \mathcal X$ shifts the point $x=(x_0,x_1,x_2,\dots)$  to the left by one coordinate, $(\sigma x)_i=x_{i+1}$, $i\ge0$.

We consider a shift invariant (or stationary) probability measure $\mu$   on $(\mathcal X,\mathcal F)$. 
For any  $A\in\mathcal F$ of positive measure,    $\mu_{A}(\cdot):=\frac{\mu(\{x\in \cdot\cap A\})}{\mu(A)}$ is the conditional measure $\mu$ restricted to $A$. 


Our results are stated under two mixing conditions that we now define.
For all $n\geq 1$, define
\begin{align*}
\phi(n)&:=\sup_{i\in\mathbb N, A\in\mathcal F_0^i,B\in\mathcal F_{i+n}^\infty}\left|\frac{\mu(A\cap B)}{\mu(A)}-\mu(B)\right| ,\\
\psi(n)&:=\sup_{i\in\mathbb N, A\in\mathcal F_0^i,B\in\mathcal F_{i+n}^\infty}\left|\frac{\mu(A\cap B)}{\mu(A)\mu(B)}-1\right| .
\end{align*}

Note that $\psi(n)$ and $\phi(n)$ are nonincreasing sequences, since $\mathcal{F}_0^{i}\subset\mathcal{F}_0^{i+1}$ for every $i\geq 0$.

\begin{defi}\label{def:mixing}
We say that the measure $\mu$ on $(\mathcal X,\mathcal F)$ is 
$\phi$-mixing (\emph{resp.} $\psi$-mixing) if 
$\phi(n)$ (\emph{resp.}  $\psi(n)$) goes to $0$ as $n$ diverges. We will say that $\mu$ is ``summable $\phi$-mixing'' if it is $\phi$ mixing with $\sum_n\phi(n)<\infty$.
\end{defi}

We refer to \cite{bradleysurvey} for an exhaustive review of mixing properties and examples. 

\subsection{Recurrence times and exponential approximations}\label{sec:not}

The hitting time of a point $y$ to a  set $A\in\mathcal F$ is defined by
\begin{align*}
T_{A}(y)&=\inf\{k\ge1:\sigma^k(y)\in A\}.
\end{align*}

For sets $A$ of small measure (rare events), and under mixing conditions such as the ones introduced in the preceding subsection, it is expected that $\mu(T_A>t)$ is approximately exponentially distributed. This is what we call hitting time exponential approximation. Similarly, when we refer to return time, we mean that  we study the approximation of  $\mu_A(T_A>t)$, that is, the measure of the same event, conditioned on the points starting in $A$. 

In this paper we are interested in the case where we fix \emph{any} point $x$ and consider  $A_n(x)$ as target set. When $n$ diverges, the measure of $A_n(x)$ vanishes, leading to rare events. The scaling parameter of the exponential approximation will depend on the point $x$. 
 
The two main types of approximations that appeared in the literature when approximating the hitting/return time distributions around any point 
$x$ of the phase space are a total variation distance type and a pointwise type.

\begin{itemize}
\item \emph{{\bf Type 1}:  Total variation distance.} For any $x\in \mathcal{X}$, 
\begin{itemize}
\item Hitting times
\[
\sup_{t>0}\left|\mu(T_{A}>t)-e^{-\mu(A)\theta(A)t}\right|\le\epsilon(A), 
\] 
\item Return times
\[
\sup_{t>0}\left|\mu_{A}(T_{A}>t)-\bar\theta(A)e^{-\mu(A)\theta(A)t}\right|\le\epsilon(A).
\] 
\end{itemize}

\item \emph{{\bf Type 2}:  Pointwise.} For any $x \in \X$ and any $t>0$, 
\begin{itemize}
\item Hitting times
\[
\left|\mu(T_{A}>t)-e^{-\mu(A)\theta(A)t}\right|\le\epsilon(A,t) ,
\] 
\item Return times
\[
\left|\mu_{A}(T_{A}>t)-\bar\theta(A)e^{-\mu(A)\theta(A)t}\right|\le\epsilon(A,t) .
\] 

\end{itemize}
\end{itemize}
Note that in the return time approximation, the parameters $\theta$ and $\bar\theta$ need not to be equal. However, such approximation leads to 
\[
\mathbb{E}_{A}(T_A)\approx \bar\theta(A)\frac{1}{\mu(A)\theta(A)}.
\]
In view of Kac Lemma which, we recall, states that $\mathbb{E}_{A}(T_A)=\frac{1}{\mu(A)}$, the last display suggests that $\theta$ and $\bar\theta$ must be close. 

\subsection{Potential well: definition and genealogy in PRT}\label{secgenealogy}

As already explained, \emph{potential well} will be used as scaling parameter in the exponential approximations of Types 1 and 2 defined above. 
In order to define it, we need first to define the \emph{shortest possible return} of  a  set $A\in \mathcal F$ (to itself)
\begin{equation*}\label{eq:tau_smallest}
\tau(A):= \inf_{y\in A} \left\lbrace T_A(y):\mu_A\left(\sigma^{-T_A(y)}(A)\right)>0\right\rbrace,
\end{equation*}
or, equivalently
\[
\tau(A):=\inf\left\lbrace k\geq 1: \mu_A\left(\sigma^{-k}(A)\right)>0\right\rbrace.
\]

In the case where $A=A_n(x)$, we can define $\tau_n(x)=\tau(A_n(x))$, and the $\tau_n:\mathcal X\rightarrow \mathbb N$ constitutes a sequence of simple functions.

%
%
%

The first possible return time $\tau_n(x)$ is an object of independent interest which was studied under several perspectives in the literature. Let us mention that its asymptotic concentration was proved by \cite{saussol/troubetzkoy/vaienti/2002} and \cite{afraimovich/chazottes/saussol/2003}, large deviations in \cite{abadi/vaienti/2008}, \cite{haydn/vaienti/2010} and \cite{abadi/cardeno/2015}, and fluctuations in \cite{abadi/lambert/2013} and \cite{abadi2017shortest}.


Obviously, by definition, $\mu_A(T_A \ge\tau(A))=1$. If for a point $x \in A$ we have $T_A(x)>\tau(A)$, we say that $x$ \emph{escapes} from $A$.
The \emph{potential well} of order $n$ at $x$ is precisely the proportional measure of points of $A$ which escape from $A$
\begin{equation*}\label{eq:pot_well}
\rho(A):=\mu_{A}(T_{A} >\tau(A))  .
\end{equation*}
Since we are interested in the case where $A=A_n(x)$, we may use the alternative notation $\rho(x_0^{n-1})$.
Besides being explicitly computable in many situations, the potential well is physically meaningful and, as scaling parameter, provides precise exponential approximations for recurrence times under suitable mixing assumptions. 

\vspace{0.2cm}

We give below a small genealogy of scaling parameters that appeared in the literature that consider results holding \emph{for all} points to get approximations for hitting/return times. 
\begin{itemize}
\item As far as we know the first paper to prove exponential approximations for hitting time statistics \emph{for all points} is due to Aldous and Brown \cite{aldous1993inequalities}. They obtained Type 1 approximations in the case of  reversible Markov chains. The parameter used there is just  the inverse of the expectation, which is mandatory to use when the approximating law is the exponential distribution. 
However, this does not bring information about the value of the expectation.

\item Galves and Schmitt \cite{galves/schmitt/1997} obtained  Type 1 approximations for hitting times in $\psi$-mixing processes. The major breakthrough there was that the authors provided an explicit formula for the parameter (denoted by $\lambda(A)$). This quantity could be viewed as the \emph{grandfather} of $\rho$. Nonetheless, its explicit significance was not evident.

\item
\cite{abadi2001exponential} and \cite{abadi/2004}  gave  exponential approximations (Type 1 and Type 2 respectively) of the distribution of hitting time around {any} point using a scaling parameter. In \cite{abadi/2004} however, only its existence and necessity were proven, the calculation of $\lambda$ being intractable in general.  The main problem is that $\lambda(A)$ depends on the recurrence property of the cylinder $A$ up to  large time scales  (usually of the order of $\mu(A)^{-1}$). 

\item In order to circumvent this issue, \cite{abadi2001exponential} also provided,   in the context of approximations of Type 1, another scaling parameter, easier to compute, but with a slightly larger error term as a price to pay. 
It is defined as follows
\[
\zeta_{s}(x_{0}^{n-1}):=\mu_{x_0^{n-1}}(T_{x_0^{n-1}}>n/s).
\]
This quantity  depends on, at most, the $2n$ first coordinates of the process. $\zeta_{s}(x_{0}^{n-1})$ can be seen as the \emph{father} of the potential well. Both works \cite{abadi2001exponential} and \cite{abadi/2004} lead with processes enjoying $\psi$-mixing or summable $\phi$-mixing. 

\item The use of the potential well $\rho$ as scaling parameter was firstly proposed by Abadi in  \cite{abadi2006hitting}, still in the context of an approximation of Type 1 for hitting and return times. More specifically, it is proved 
that, for exponentially  $\alpha$-mixing processes, $\lambda$ and $\zeta$ (grandpa and father of $\rho$) can be well approximated by $\rho$.

\item The first paper to really directly use $\rho$ as scaling parameter was \cite{abadi/vergne/2009}, in which a type 2 approximation for return times was obtained, with $\bar\theta=\theta=\rho$. The process is assumed to be $\phi$-mixing.

\item Focusing on proving exponential approximations for hitting and return times under the largest possible class of systems, and still for all points, 
abadi and Saussol \cite{abadi2011hitting} returned to the approach of Galves and Schmitt. Their results hold under the $\alpha$-mixing condition, which is the weakest hypothesis  used up to date, but the scaling parameter is not explicit. 

\item Focussing on the specific class of binary renewal processes,  \cite{abadi/cardeno/gallo/2015}  proved  a type 1 approximation for hitting and return times using the potential well $\rho$. One interesting aspect  concerning this work lays in the fact that the renewal process is $\beta$-mixing (weaker than the $\phi$-mixing assumed by \cite{abadi/vergne/2009}). Moreover, the authors managed to use the renewal property to compute the limit of $\rho(A_n(x))$ for \emph{any} point $x$. In other words, the approximating asymptotic law for hitting and return times was  explicitly computed  as function of the parameters of the process. This result shows the usefulness of the potential well, an ``easy to compute'' scaling parameter.

\end{itemize}

%
%
%
%
%
%
%

\section{Main  results}\label{sec:results}

Theorem \ref{theo:correct} below presents Type 2 approximations for hitting and return time under $\phi$ and $\psi$-mixing conditions with the potential well as scaling parameter and an explicit error term. 

Before we can state this result, we first need to define the second order periodicity of string $A_n(x)$, which plays a crucial role for the size of the error term. 

\subsection{Second order periodicity}

The short returns that we will  define here are precisely those that are difficult to treat as (almost) independent. They not only depend on the correlation decay of the system but also on the particular properties of the string itself. Technically, for an $n$-cylinder, \emph{short} means returning in up to the order $n$ steps.

Consider the cylinder $A$ and suppose   $\tau(A)=k$. Write $n=qk+r$, where $q\in\mathbb{N}$ and $0\leq r<k$, 
and note that the cylinder overlaps itself in all multiples of $k$ smaller than $n$. The set $\mathcal P(A):=\{mk:1\leq m\leq q\}$ are  indexes of possible returns at multiples of $\tau(A)$, but returns can also occur at other time indexes after that. Let
$$
\mathcal{R}(A)=\{j\in\{qk+1,...,qk+r-1\}:\,    \mu_A(\sigma^{-j}(A)) >0 \}.
$$

A point $y\in A$ could only return to $A$ before $n$ at time indexes in $\mathcal{P}(A) \cup \mathcal{R}(A)$, but there is  a crucial difference between them. A point that escapes from $A$ can not return in $\mathcal{P}(A)$, but it \emph{could} return in $\mathcal{R}(A)$.
Namely, 
\[
\mu_A( \sigma^{-\tau(A)}(A^c) \cap   \sigma^{-j}(A) )   =0, \ \  j\in \mathcal{P}(A),
\]
while
\[
\mu_A( \sigma^{-\tau(A)}(A^c) \cap   \sigma^{-j}(A) ) >0, \ \  j\in \mathcal{R}(A).
\]

We set $n_A$ as the first possible return to $A$, \emph{among those points   $x \in A$ who escape $A$ at $\tau(A)$ } 
$$n_A =
\left \{
\begin{array}{lc}
\min\mathcal{R}(A)&\mathcal{R}(A)\neq\emptyset \\
\min\{j  :   \mu_A( \sigma^{-\tau(A)}(A^c) \cap   \sigma^{-j}(A) ) >0  \} ,                      & \mathcal{R}(A)=\emptyset. \\
\end{array}
\right.
$$
Actually, in the second case,  $n_A\ge n$. 
We refer to \cite{abadi/vergne/2009} to find an example that illustrates these facts.
This definition is slightly more general that the one therein, since we include the case of a non complete grammar\footnote{We say $\mu$ has complete grammar if $\mu(A_n(x))>0$ for any $x\in\mathcal X$ and $n\ge1$.}.

\subsection{Type 2 approximations scaled by the potential well}\label{sec:correct}

For any finite string $A$, let us denote with  $A^{(k)}$  the suffix of $A$ of size $k$. That is, if $A=x_0^{n-1}$, then $A^{(k)}=x_{n-k}^{n-1}$. When $A$ is an $n$-cylinder we will use the convention 
$\mu( A^{(j)}) = \mu( A^{(n)})=\mu(A)$ for $j\geq n$.

By definition,  $\phi(g)$ is finite for all $g\ge 1$. This is not the case of $\psi(g)$.
Thus, for $\psi$-mixing measures, we define 
\begin{equation}\label{def:g0}
g_0=g_0(\psi):=\inf\{g\geq 1:\psi(g)<\infty\}-1.
\end{equation}

Now, for the error term, define
\begin{align}
& \label{epspsi}(a)\;\;  \epsilon_{\psi}(A):=n\mu\left(A^{(n_A-g_0)}\right)+\psi(n),\\
& \label{epsphi}(b) \;\;\epsilon_{\phi}(A):=\inf_{1\leq w\leq n_A}\left\lbrace (n+\tau(A))\mu\left(A^{(w)}\right)+\phi(n_A-w)\right\rbrace . 
\end{align}
Note that cylinders $A$ of size $n$ verify that $n_A\geq n/2$, then $\epsilon_{\psi}$ is well defined for all $n> 2g_0$.

We will  use  $\epsilon$ to denote either $\epsilon_\psi$ or $\epsilon_\phi$ when the argument/statement is general.




\begin{theo}\label{theo:correct}
Consider a stationary measure $\mu$ on $(\mathcal X,\mathcal F)$ enjoying either $\phi$-mixing with $\sup_{A\in\mathcal{A}^n}\mu(A)\tau(A)\stackrel{n}\longrightarrow 0$, or simply $\psi$-mixing. There exist five positive constants
$C_i$, $i=1,\ldots,5$, and  $n_0\in\N$ such that for all $n\geq n_0$ and all $A\in\mathcal{A}^n$, the following inequalities hold.
\begin{enumerate}
\item For all $t\geq 0$:
\begin{equation*}
\hspace{-1cm}\left|\mu(T_A>t)-e^{-\rho(A)\mu(A)t}\right|\leq \left \{
\begin{array}{lc}
C_1\left(\tau(A)\mu(A)+t\mu(A)\epsilon(A)\right) & t\leq [2\mu(A)]^{-1}   \\
C_2\,\mu(A)t\epsilon(A)e^{-\mu(A)t\left(\rho(A)-C_3\epsilon(A)\right)} & t>[2\mu(A)]^{-1} .\\
\end{array}
\right.
\end{equation*}
\item For all $t\geq \tau(A)$:
\begin{equation*}
\left|\mu_A(T_A>t)-\rho(A)e^{-\rho(A)\mu(A)(t-\tau(A))}\right|\hspace{4cm}
\end{equation*}
\begin{equation*}
\hspace{1cm}\leq \left \{
\begin{array}{lc}
C_4\,\epsilon(A) & t\leq [2\mu(A)]^{-1} \\
C_5\,\mu(A)t\epsilon(A)e^{-\mu(A)t\left(\rho(A)-C_3\epsilon(A)\right)} & t>[2\mu(A)]^{-1} .\\
\end{array}
\right.
\end{equation*}
\end{enumerate}
\end{theo}

\vspace{0.2cm}

Theorem \ref{theo:correct} (and its proof) are definitely inspired by \cite{abadi/vergne/2009} and their  Theorem 4.1. However, let us first observe that our result provides the first statement of the literature for Type 2 hitting time approximations with the potential well as scaling parameter. Moreover, contrarily to \cite{abadi/vergne/2009},  we do not assume complete grammar nor finite alphabet. 

\vspace{0.2cm}
Let us make some further important observations concerning this theorem.

\begin{remark}
Under $\phi$-mixing, the assumption $\sup\mu(A)\tau(A)\stackrel{n}\longrightarrow 0$ can be dropped under certain circumstances.   For instance, if the measure has complete grammar, we have $\tau(A)\leq n$ and the assumption is granted using Lemma \ref{lemsubexp}. Another way is to assume that $\mu$ is \emph{summable} $\phi$-mixing or $\psi$-mixing, as commented after Lemma \ref{lemtau} in Section \ref{sec:proofs}.
\end{remark}

\begin{remark}\label{rem.error}
According to Lemma \ref{lemsubexp}, if $\mu$ is $\phi$-mixing (and \emph{a fortiori}, $\psi$-mixing), there exist constants $C$ and $c$ such that $\mu(A)\leq Ce^{-cn}$ for all $n\geq 1$ and $A\in\mathcal{A}^n$.  On the other hand, since  $n_A\geq n/2$, we get 
 $\mu\left(A^{(n_A-g_0)}\right)\leq Ce^{-c(n/2-g_0)}$ for all $n>2g_0$. Therefore, $\epsilon_{\psi}(A)\stackrel{n}\longrightarrow 0$ uniformly.
Further, if $\tau(A)\leq 2n$ it is enough to take $w=\lceil n/4\rceil$ to obtain $\phi(n_A-w)\leq \phi(\lfloor n/4\rfloor)$ and $(n+\tau(A))\mu\left(A^{(w)}\right)\leq 3Cne^{-cn/4}$, which ensures $\epsilon_{\phi}(A)\stackrel{n}\longrightarrow 0$ uniformly. This is the case, for instance,  if one has  complete grammar.
On the other hand, notice that $\tau(A)<n_A$. Hence, if $\tau(A)>2n$ we take $w=n$ and get $\phi(n_A-w)\leq\phi(n)$. Therefore, since $\tau(A)\mu(A)\stackrel{n}\longrightarrow 0$, we also have in this case $\epsilon_{\phi}\stackrel{n}\longrightarrow 0$. 
\end{remark}

\begin{remark}
Naturally, the statements under $\psi$-mixing are less general, but have smaller error terms. 
The error term is the same for $t>[2\mu(A)]^{-1}$ for both hitting and return times approximations. 
The difference is for small $t$'s, due to the correlation arising from the conditional measure. 
\end{remark}

\begin{remark}
For application purposes involving data, it is essential to control all the constants involved in the statements. These constants can be accessed from the proof presented in Section \ref{subsectheo1}, {where we also make explicit the integer $n_0$ from which Theorem \ref{theo:correct} holds} (see \eqref{defn0}). If $\mu$ is $\psi$-mixing, we define $M:=\psi\left(g_0+1\right)+1$. In this case $C_1=8M+9$, $C_2=194M+206$, $C_3=66M+89$, $C_4=12M+15$ and $C_5=197M+220$. On the other hand, for the $\phi$-mixing case we have $C_1=9$, $C_2=143$, $C_3=61$, $C_4=14$ and $C_5=170$.
\end{remark}

\begin{remark}
We show the sharpness of the error term in the return time approximation given by  Theorem \ref{theo:correct} with a simple example.  Consider an i.i.d. process $(X_m)_{m\in\mathbb{N}}$ with alphabet $\mathcal{A}$. Take $b\in\mathcal{A}$ such that $\mu(b)=p$ and $x=(b,b,\ldots)\in\mathcal{X}=\mathcal{A}^{\mathbb{N}}$.
Thus $A_n(x)=x_0^{n-1}=(b,b,\ldots,b)$. 
Direct calculations give
\begin{itemize}
\item $\mu(A_n)=p^{n}$
\item $\tau(A_n)=1$
\item $\displaystyle\rho(A_n)=1-\mu_{A_n}(T_{A_n}=\tau(A_n))=1-\mu_{A_n}(X_n=b)=1-p$
\item $n_{A_n}=n$.
\end{itemize}
An i.i.d. process is trivially $\psi$-mixing with function $\psi$ identically zero. Thus Theorem \ref{theo:correct} states that the error for small $t$'s is  $\epsilon_\psi(A_n) = np^n$. 
On the other hand, by direct calculation we have for each $n\geq 2$

\begin{align*}
\left|\mu_{A_n}(T_{A_n}>n-1)-\rho(A_n)e^{-\rho(A_n)\mu(A_n)((n-1)-\tau(A_n))}\right|=(1-p)\left(1- e^{-(1-p)p^n(n-2)} \right)
\end{align*}
which implies that the exact error in the approximation for return time at $n-1$ is of order  $p^n n$, just as stated by Theorem \ref{theo:correct}.
\end{remark}

\begin{remark}
The reader may notice a difference between Theorem \ref{theo:correct} and  Theorem 4.1 of \cite{abadi/vergne/2009} concerning  the error term for small $t$'s for return time approximation. Indeed, their statement is incorrect as shown by the preceding example. We recall that the error term for small $t$'s  plays a fundamental role when  studying  return time spectrum, as was done by \cite{abadi2019complete}. Theorem \ref{theo:correct}, besides correcting \cite{abadi/vergne/2009} is also fundamental to correct \cite{abadi2019complete} which was based on the exponential approximations given by \cite{abadi/vergne/2009}. 
\end{remark}


 \subsection{Uniform positivity of the potential well}\label{sec:conv}

Theorem \ref{theo:correct} says  that the potential well can be  used  as scaling parameter to obtain approximations for recurrence times around \emph{any} point. We now ask about the possible values of this scaling parameter in its range $[0,1]$.

 Abadi and Saussol \cite{abadi2016almost}, in the more general case known up to now,
proved that for $\alpha$-mixing  processes with at least polynomially decaying $\alpha$ rates, the distribution of hitting and return time converge, \emph{almost surely}, to an exponential with parameter $1$. We refer \cite{bradleysurvey} for the precise definition of $\alpha$-mixing, but the only important point for us it to know that summable $\phi$-mixing implies  $\alpha$-mixing with at least polynomially decaying $\alpha$ rates. This fact, combined with Theorem \ref{theo:correct}, proves, indirectly, that for summable $\phi$-mixing processes, the potential well converges almost surely to $1$, since both theorem must agree on the limiting distribution under these conditions. Theorem \ref{theo:positive} item (a) below states that the same holds for $\phi$-mixing without any assumption on the rate. 

On the other hand, for the renewal processes, with certain tail distribution for the inter-arrival times, Abadi, Carde\~no and Gallo \cite{abadi/cardeno/gallo/2015} proved that  for the point $x=(00000...)$, the sequence of potential wells  
$\rho(x_0^{n-1})$ converges to $0$. 
In this case, the scaling parameter has a predominant role, indicating the drastic change of  scale of occurrence of events. For instance, in this case
the mean hitting time is much larger than the mean return time
\[
\mathbb E(T_{x_0^{n-1}})\approx\frac{1}{\rho(x_0^{n-1})\mu(x_0^{n-1})}\gg\frac{1}{\mu(x_0^{n-1})}=\mathbb E_{x_0^{n-1}} (T_{x_0^{n-1}}).
\]
Such renewal processes are $\beta$-mixing (see \cite{bradleysurvey} for the definition). Theorem \ref{theo:positive} item (b) below  states that this cannot happen for $\psi$-mixing processes or summable $\phi$-mixing processes.

 \begin{theo}\label{theo:positive}
Let $\mu$ be a stationary $\phi$-mixing measure. Then
\begin{itemize}
\item[(a)] $\rho(x_0^{n-1})\stackrel{n}\longrightarrow1$, almost surely.
\item[(b)] If $\mu$ is $\psi$-mixing or summable $\phi$-mixing, there exists $n_1\geq 1$ such that: 
\[
\inf_{n\geq n_1,x_{0}^{n-1}\in\mathcal{A}^n}\rho(x_{0}^{n-1})= \rho_->0\,.\nonumber
\]
\end{itemize}
\end{theo}
 
If the alphabet $\mathcal{A}$ is finite, the set $\{\rho\left(A\right): A\in\mathcal{A}^n,n<n_1\}$ is finite and has a strictly positive infimum, which implies that the infimum above can be taken over all $n\geq 1$.






\section{Proofs of the results}\label{sec:proofs}

The statement of Theorem \ref{theo:correct} is for $\phi$ and $\psi$ and  for hitting and return times. 
The case of return times under $\phi$-mixing was already done by \cite{abadi/vergne/2009}. 
Our proof follows their method. 
In particular for the next subsection that lists a sequence of  auxiliary results, some of them will not be proved.

\subsection{Preliminary results} 

The following lemma plays a fundamental role in Theorems \ref{theo:correct} and \ref{theo:positive}. 
It was originally proved in  \cite{abadi2001exponential} assuming summability of the function $\phi$, an assumption which can be  dropped. 

\begin{lemma}\label{lemsubexp}
Let $\mu$ be a $\phi$-mixing measure. Then, there exists positive constants $C$ and $c$ such that for all $n\geq 1$ and all $A\in\mathcal{A}^n$, 
one has:
$$\mu(A)\leq Ce^{-cn}.$$
\end{lemma}
\begin{proof}
We denote by $\lambda=\sup\{\mu(a):a\in\mathcal{A}\}<1$. Consider a positive integer $k_0$ and for all $n \ge k_0$ write $n=k_0q+r$, with $1\leq q\in\mathbb{N}$ and $0\leq r<k_0$. Suppose $A=a_0^{n-1}$, and apply the $\phi$-mixing property to obtain:
\begin{align*}
\mu(A)&\leq  \mu\left(\bigcap_{j=0}^{q-1}\left\lbrace \sigma^{-jk_0}(a_{jk_0})\right\rbrace\right)\leq \mu\left(\bigcap_{j=0}^{q-2}\left\lbrace \sigma^{-jk_0}(a_{jk_0})\right\rbrace\right)(\phi(k_0)+\mu(a_{(q-1)k_0}))\\
&\leq  \mu\left(\bigcap_{j=0}^{q-2}\left\lbrace \sigma^{-jk_0}(a_{jk_0})\right\rbrace\right)(\phi(k_0)+\lambda) .
\end{align*}
Iterating this argument one concludes
\[
\mu(A) \leq  (\phi(k_0)+\lambda)^{q} .
\]
Since $\phi(k)\stackrel{k}\longrightarrow 0$, there exists $k_0\in\mathbb{N}$ such that $\phi(k_0)+\lambda<1$. Thus, for  $n\geq k_0$, and observing that  $q= \frac{n-r}{k_0}>\frac{n}{k_0}-\frac{k_0-1}{k_0}$
$$\mu(A)\leq  (\phi(k_0)+\lambda)^{-(k_0-1)/k_0}        \left(\left(\phi(k_0)+\lambda\right)^{1/k_0}\right)^{n}.$$
This covers the case $n\ge k_0$. By eventually enlarging the constant $C$, one covers the case $n<k_0$. This ends the proof. 

\end{proof}

Under the assumption of complete grammar, we would obviously have, by definition, $\tau(A_n(x))\le n$. But since we do not assume this, we need the following lemma, which provides upper bounds for $\tau(A)$ when $\mu$ is $\psi$-mixing or summable $\phi$-mixing. 

\begin{lemma}\label{lemtau}
Consider $\mu$ a $\psi$-mixing or summable $\phi$-mixing measure. Then, there exists $n_2\in\N$ such that for all $n\geq n_2$ and $A\in\mathcal{A}^n$, we have
\begin{itemize}
\item $\tau(A)\leq 2n$,\hspace{0.5cm} for $\psi$;
\item $\displaystyle \tau(A)\leq -\frac{2}{\mu\left(A\right)\ln\mu\left(A\right)}+n$,\hspace{0.5cm} for summable $\phi$.
\end{itemize}
\end{lemma}
\begin{proof}
We start with the case $\psi$. For $n$ large enough we have $\psi(n)<1$, which implies:
\[
\mu\left(A\cap \sigma^{-2n}(A)\right)\geq \mu(A)^2(1-\psi(n))>0
\]
Since $\tau(A)$ is the smallest positive integer such that $\mu\left(A\cap \sigma^{-\tau(A)}(A)\right)>0$, we must have $\tau(A)\leq 2n$.

Now consider the $\phi$-mixing case.  Summability of $\phi$  ensures that for $g$ large enough we have $\phi(g)\leq 1/(g\ln g)$. Thus
\[
\mu\left(A\cap\sigma^{-g-n}\left(A\right)\right)\geq \mu(A)\left(\mu\left(A\right)-\phi(g)\right)\geq \mu(A)\left(\mu\left(A\right)-\frac{1}{g\ln g}\right).
\]
Take $\displaystyle g=-\frac{2}{\mu\left(A\right)\ln\mu\left(A\right)}$. The rightmost parenthesis above becomes
\[
\mu\left(A\right)\left[1-\frac{1}{2}\frac{-\ln\mu\left(A\right)}{(\ln(2)-\ln\mu\left(A\right)-\ln(-\ln\mu\left(A\right)))}\right]
\]
which is positive for $n$ large enough.
\end{proof}

The multiplicative constant $2$ in both cases is technical and was chosen for the simplicity of the proof. Actually, it can be replaced by any constant strictly larger than one.
An irreducible aperiodic finite state Markov chain with some entry equal to zero shows that this constant can not be taken equal to one
in the $\psi$-mixing case.  Whether this bound is optimal for the $\phi$-mixing case is an open question.
Note that Lemmas \ref{lemsubexp} and \ref{lemtau} imply that $\tau(A)\mu(A)\stackrel{n}\longrightarrow 0$ uniformly.

The remaining results of this subsection hold for $n\ge n'$, where $n'=1$ for the case of $\phi$-mixing and
\begin{align}
\label{defn'} n':=\inf\{n>2g_0: \psi(n)<1\}
\end{align}
for the $\psi$-mixing case (see \eqref{def:g0} for the definition of $g_0$).


%
Let us define
\[
M:=\psi\left(g_0+1\right)+1 .
\]

\begin{prop}\label{R(A)}
Let $\mu$ be a $\psi$-mixing measure. Then for all $n\ge n'$, $A\in\mathcal{A}^n$ and  $k\geq n$, the following inequalities hold
\begin{itemize}
\item[(a)] $\mu_A({T_A}\in\mathcal{R}(A))\leq M \,|\mathcal{R}(A)|\, \mu\left(A^{(n_A-g_0)}\right)$
\item[(b)] $\mu_A(n\leq{T_A}\leq k)\leq M(k-n+1)\mu\left(A^{(n-g_0)}\right)$.
\end{itemize}
\end{prop}
\begin{proof}
(a) We consider the case $\mathcal R(A)\not=\emptyset$, otherwise it is trivial.
For all $j\geq 1$ and $n\geq n'$, one trivially has
\[
\{{T_A}=j\}\subset \sigma^{-j}(A)\subset \sigma^{-j-(n-(n_A-g_0))}\left(A^{(n_A-g_0)}\right).
\]
Thus
\begin{align*}
\mu_A({T_A}\in \mathcal{R}(A))
\leq \mu_A\left(\bigcup_{j\in \mathcal{R}(A)}  \sigma^{-j-(n-(n_A-g_0))}\left( A^{(n_A-g_0)}\right)  \right).
\end{align*}
Note that  $A$ and the union on the right hand side in the above inequality 
are separated by a gap of length $g_0+1$. By $\psi$-mixing, the left hand side is bounded  by
\begin{align*}
(\psi(g_0+1)+1)  \mu\left(  \bigcup_{j\in \mathcal{R}(A)}  \sigma^{-j-(n-(n_A-g_0))} \left( A^{(n_A-g_0)} \right)  \right)  
\leq  M\, |\mathcal{R}(A)|\,  \mu\left(A^{(n_A-g_0)}\right).
\end{align*}

(b) In a similar way to item (a)
\begin{align*}
\mu_A(n\leq{T_A}\leq k)&
\leq \mu_A\left(\bigcup_{n\leq j\leq k}   \sigma^{-j-g_0}\left(A^{(n-g_0)}\right) \right)\\
&\leq M(k-n+1)\mu\left(A^{(n-g_0)}\right).
\end{align*}
\end{proof}

For the next proposition, recall that $\epsilon$ stands either for $\epsilon_\psi$ (\ref{epspsi}) or for $\epsilon_\phi$ (\ref{epsphi}), according to the mixing property of the measure under consideration. 
Further, let us use the notation ${T_A}^{[i]}:=T_A\circ\sigma^i$. 



 \begin{prop}\label{PR}
Let $\mu$ be a $\phi$ or $\psi$-mixing measure. Then for all $n\ge n'$, $A\in\mathcal{A}^n$ and $t\geq\tau(A)$  
$$|\mu_A({T_A}>t)-\rho(A)\mu({T_A}>t)|\leq C\epsilon(A),$$
where $C=4$ for $\epsilon_{\phi}$ and $C=4(M+1)$ for $\epsilon_{\psi}$.
\end{prop}
\begin{proof}
The proof for $\epsilon_{\phi}$ can be found in \cite[Proposition 4.1 item (b)]{abadi/vergne/2009}. We observe that the error term defined therein is 
$$
\epsilon'(A)=\displaystyle\inf_{1\leq w\leq n_A}\left\lbrace(2n+\tau(A))\mu\left(A^{(w)}\right)+\phi(n_A-w))\right\rbrace\leq 2\epsilon_{\phi}(A), 
$$ 
which justifies $C=4$ for this case. Here we prove the case  $\epsilon_{\psi}$ in the same way. We start assuming that $t\geq\tau(A)+2n$. By the triangle inequality
\begin{align}
&|\mu_A({T_A}>t)-\rho(A)\mu({T_A}>t)| \nonumber\\
&\leq  \left|\mu_A\left({T_A}>\tau(A);{T_A}^{[\tau(A)]}>t-\tau(A)\right)-\right.\nonumber\\
\label{pr1}&\hspace{0.4cm}\left.\mu_A\left({T_A}>\tau(A);{T_A}^{[\tau(A)+2n]}>t-\tau(A)-2n\right)\right|\\
\label{pr2}&+\left|\mu_A\left({T_A}>\tau(A);{T_A}^{[\tau(A)+2n]}>t-\tau(A)-2n\right)-\rho(A)\mu\left({T_A}>t-\tau(A)-2n\right)\right|\\
\label{pr3}&+\left|\rho(A)\mu({T_A}>t-\tau(A)-2n)-\rho(A)\mu({T_A}>t)\right|.
\end{align}

For the first modulus, by inclusion of sets we get immediately that
\begin{align*}
\eqref{pr1}\,&\le \mu_A\left({T_A}>\tau(A);{T_A}^{[\tau(A)]}\leq 2n\right)\\
&\leq\mu_A({T_A}\in\mathcal{R}(A))+\mu_A(n\leq T_A\leq \tau(A)+2n) \\
&
\leq  M|\mathcal{R}(A)|\mu\left(A^{(n_A-g_0)}\right)+M(\tau(A)+n+1)\mu\left(A^{(n-g_0)}\right)\\
&\leq 4Mn\mu\left(A^{(n_A-g_0)}\right).
\end{align*}
The third inequality follows from Proposition \ref{R(A)} and the last one follows from Lemma \ref{lemtau}.
 
By $\psi$-mixing,  the modulus   \eqref{pr2} is bounded by 
$$\rho(A)\mu({T_A}>t-\tau(A)-2n)\psi(n)\leq \psi(n).$$

Note that the modulus is not needed for \eqref{pr3}, and by inclusion we get
\begin{align*}
\eqref{pr3}&\le \rho(A)\mu\left({T_A}^{[t-\tau(A)-2n]}\leq \tau(A)+2n\right)\\
&=\rho(A)\mu\left({T_A}\leq \tau(A)+2n\right)\\
&\leq (2n+\tau(A))\mu(A)\\
&\leq 4n\mu\left(A^{(n_A-g_0)}\right)
\end{align*}
where the equality and second inequality follow from stationarity of $\mu$.

 Therefore, for $t\geq\tau(A)+2n$, the sum of $(\ref{pr1})$, $(\ref{pr2})$ and $(\ref{pr3})$ is bounded by
$$ 4Mn\mu\left(A^{(n_A-g_0)}\right)+\psi(n)+4n\mu\left(A^{(n_A-g_0)}\right)\leq 4(M+1)\epsilon_{\psi}(A).$$

We now consider the case where $\tau(A)\leq t<\tau(A)+2n$, we have
\begin{align*}
&|\mu_A({T_A}>t)-\rho(A)\mu({T_A}>t)| \quad \\
&\quad \leq  |\mu_A({T_A}>t)-\rho(A)|+|\rho(A)-\rho(A)\mu({T_A}>t)|\\
&\quad \leq  \mu_A(\tau(A)<{T_A}\leq \tau(A)+2n)+t\mu(A)\\
&
\quad \leq  M\left(|\mathcal{R}(A)|\mu\left(A^{(n_A-g_0)}\right)+(\tau(A)+n+1)\mu\left(A^{(n-g_0)}\right)\right)+(\tau(A)+2n)\mu\left(A^{(n_A-g_0)}\right)\\
&\quad \leq 4(M+1)\epsilon_{\psi}(A).
\end{align*}
The last but one inequality follows again by Proposition \ref{R(A)}. The other inequalities are straightforward.  This ends the proof.
\end{proof}

The next lemma establishes upper bounds for the tail distribution at the scale given by Kac's Lemma, namely $1/\mu(A)$.
For technical reasons we actually choose the scale 
\begin{equation*}\label{def:fA}
f_A:= 1/(2\mu(A)).
\end{equation*}


\begin{lemma}\label{lemfat} 
Let $\mu$ be a stationary measure. Then for all $n\geq 1$, $A\in\mathcal{A}^n$, positive integer $k$ and $B\in\mathcal{F}_{kf_A}^{\infty}$ the following inequalities hold 
\begin{itemize} 
\item[(a)] $\mu({T_A}>kf_A;B)\leq ((\psi(n)+1)\mu({T_A}>f_A-2n))^k\mu(B)$,
\item[(b)] $\mu({T_A}>kf_A;B)\leq \left(\mu\left({T_A}>f_A-2n\right)+\phi(n)\right)^{k}(\mu(B)+\phi(n))$,
\item[(c)] $\mu_A({T_A}>kf_A;B)\leq (\psi(n)+1)^k\mu({T_A}>f_A-2n)^{k-1}\mu(B)$.
\end{itemize}
\end{lemma}

\begin{proof}

We start observing that $\{{T_A}>kf_A\}\subset\{{T_A}>kf_A-2n\}\in\mathcal{F}_0^{kf_A-n}$. Thus, applying the $\psi$-mixing property we get
\begin{align}
\label{lf1}\mu({T_A}>kf_A;B)\leq \mu({T_A}>kf_A-2n;B)\leq (\psi(n)+1)\mu(T_A>kf_A-2n)\mu(B)
\end{align}
Furthermore
\[
\{{T_A}>kf_A-2n\}=\left\lbrace{T_A}>(k-1)f_A;{T_A}^{[(k-1)f_A]}>f_A-2n\right\rbrace.
\]

Now one can take  in particular $B=\left\lbrace {T_A}^{[(k-1)f_A]}>f_A-2n\right\rbrace\in\mathcal F_{(k-1)f_A}^\infty$, and then  apply \eqref{lf1} with $k-1$ instead of $k$ to get
\begin{align*}
\mu({T_A}>kf_A-2n)& \leq  (\psi(n)+1)\mu\left({T_A}>(k-1)f_A-2n\right)\mu\left({T_A}^{[(k-1)f_A]}>f_A-2n\right)\nonumber\\
&=(\psi(n)+1)\mu\left({T_A}>(k-1)f_A-2n\right)\mu\left({T_A}>f_A-2n\right) . 
\end{align*}



The equality follows by stationarity. Iterating this argument one  concludes that
\begin{equation}\label{lf2}
\mu({T_A}>kf_A-2n)\leq (\psi(n)+1)^{k-1}\mu(T_A>f_A-2n)^k\;.
\end{equation}

Applying the resulting inequality in \eqref{lf1}, we get the statement (a). In a similar way,  $\phi$-mixing  gives
\begin{align*}
\mu(T_A>kf_A;B)
\leq \mu(T_A>kf_A-2n)(\mu(B)+\phi(n))
\end{align*}
And thus
\begin{align}
\mu({T_A}>kf_A-2n)   
&\leq \mu({T_A}>(k-1)f_A-2n)\left(\mu\left({T_A}^{[(k-1)f_A]}>f_A-2n\right)+\phi(n)\right)\nonumber\\
\label{lf3}&\leq \left(\mu\left({T_A}>f_A-2n\right)+\phi(n)\right)^{k}.
\end{align}
which ends the proof of (b).

The proof for (c) follows the same lines as item (a), observing that for $A,B\in\mathcal{F}_0^i$ and $C\in\mathcal{F}_{i+n}^{\infty}$, $\psi$-mixing property implies $\mu_A(B;C)\leq \mu_A(B)\mu(C)(\psi(n)+1)$.

\end{proof}


The next proposition is the key to the proof of Theorem \ref{theo:correct}, and the idea is the following. We work under the time scale $f_A$. When $t=k f_A, \ k\in\mathbb N$, then we simply cut out $t$ into $k$ pieces of equal size $f_A$. Then, the case of general $t=k f_A+r,r<f_A$ is approximated  by  its "integer part" $k f_A$. Technically, this is done in  $b)$ and $a)$ respectively. 

\begin{prop}\label{prfat}
Let $\mu$ be a $\phi$ or $\psi$-mixing measure. Then for all $n\geq n'$, $A\in\mathcal{A}^n$ and positive integer $k$, the following inequalities hold:\\
(a) For   $0\leq r\leq f_A$  
\begin{itemize}
\item[1.] $|\mu({T_A}>kf_A+r)-\mu({T_A}>kf_A)\mu({T_A}>r)|\leq C'(\psi(n)+1)^{k-1}\mu({T_A}>f_A-2n)^k\epsilon_\psi(A)$
\item[2.] $|\mu({T_A}>kf_A+r)-\mu({T_A}>kf_A)\mu({T_A}>r)|\leq C'(\mu({T_A}>f_A-2n)+\phi(n))^{k}\epsilon_\phi(A)$
\item[3.] $|\mu_A({T_A}>kf_A+r)-\mu_A({T_A}>kf_A)\mu({T_A}>r)|\leq C'((\psi(n)+1)\mu({T_A}>f_A-2n))^{k-1}\epsilon_\psi(A)$.
\end{itemize}
(b) For $k\ge1$
\begin{itemize}
\item[1.] $\left|\mu({T_A}>kf_A)-\mu({T_A}>f_A)^k\right|\leq C'\epsilon_{\psi}(A)(k-1)(\psi(n)+1)^{k-2}\mu({T_A}>f_A-2n)^{k-1}$
\item[2.] $\left|\mu({T_A}>kf_A)-\mu({T_A}>f_A)^k\right|\leq C'\epsilon_{\phi}(A)(k-1)(\mu({T_A}>f_A-2n)+\phi(n))^{k-1}$
\item[3.] $\left|\mu_A({T_A}>kf_A)-\mu_A(T_A>f_A)\mu({T_A}>f_A)^{k-1}\right|\leq C'\epsilon_{\psi}(A)(k-1)((\psi(n)+1)\mu({T_A}>f_A-2n))^{k-2}$
\end{itemize}
where $C'=2(M+1)$ for the cases involving $\psi$ and $C'=4$ for $\phi$.
\end{prop}

\begin{proof} 

We will proof items (a)-1 and (a)-2 together. Initially, consider the case in which $r<2n$. In this case, for all $n\geq n'$ we have
\begin{align}
&|\mu({T_A}>kf_A+r)-\mu({T_A}>kf_A)\mu({T_A}>r)|\nonumber\\
&\leq  \left|\mu\left({T_A}>kf_A,{T_A}^{[kf_A]}>r\right)-\mu({T_A}>kf_A)\right|+\mu({T_A}>kf_A)|1-\mu({T_A}>r)|\nonumber\\
&\label{prfat1}\leq  \mu\left({T_A}>kf_A,{T_A}^{[kf_A]}\leq r\right)+\mu({T_A}>kf_A)\mu({T_A}\leq r).
\end{align}

By Lemma \ref{lemfat}-(a) and \eqref{lf2}, the last sum is bounded by
\begin{align}
& ((\psi(n)+1)\mu(T_A>f_A-2n))^k\mu\left({T_A}^{[kf_A]}\leq r\right)+\mu({T_A}>kf_A-2n)r\mu(A)\nonumber\\
& \leq ((\psi(n)+1)\mu(T_A>f_A-2n))^k\mu\left({T_A}\leq r\right)+(\psi(n)+1)^{k-1}\mu(T_A>f_A-2n)^kr\mu(A)\nonumber\\
& \leq (\psi(n)+1)^{k-1}\mu(T_A>f_A-2n)^kr\mu(A)(M+1)\nonumber\\
& \label{pr29}\leq 2(M+1)(\psi(n)+1)^{k-1}\mu(T_A>f_A-2n)^k\epsilon_{\psi}(A)
\end{align}
which gives us (a)-1. To get (a)-2 for $r<2n$, we apply Lemma \ref{lemfat}-(b) and \eqref{lf3} in a similar way. Thus, \eqref{prfat1} is bounded by
\begin{align*}
&(\mu(T_A>f_A-2n)+\phi(n))^k(\mu(T_A\leq r)+\phi(n))+(\mu(T_A>f_A-2n)+\phi(n))^k\mu(T_A\leq r)\\
&\leq 4(\mu(T_A>f_A-2n)+\phi(n))^k\epsilon_{\phi}(A).
\end{align*}

We now consider the case $r\geq 2n$. The triangle inequality gives us
\begin{align}
& |\mu({T_A}>kf_A+r)-\mu({T_A}>kf_A)\mu({T_A}>r)|\nonumber \\
\label{pr23}&\leq  \left|\mu\left({T_A}>kf_A;{T_A}^{[kf_A]}>r\right)-\mu\left({T_A}>kf_A;{T_A}^{[kf_A+2n]}>r-2n\right)\right|\\
\label{pr24}&+\left|\mu\left({T_A}>kf_A;{T_A}^{[kf_A+2n]}>r-2n\right)-\mu\left({T_A}>kf_A\right)\mu\left({T_A}^{[kf_A+2n]}>r-2n\right)\right|\\
\label{pr25}&+\left|\mu\left({T_A}>kf_A\right)\mu\left({T_A}^{[kf_A+2n]}>r-2n\right)-\mu\left({T_A}>kf_A\right)\mu\left({T_A}>r\right)\right|.
\end{align}


We proceed as in \eqref{pr1} and use Lemma \ref{lemfat}-(a) to get
\begin{align}
\eqref{pr23}&\le\mu\left({T_A}>kf_A;{T_A}^{[kf_A]}\leq 2n\right)\nonumber\\
& \leq ((\psi(n)+1)(\mu(T_A>f_A-2n))^k\mu(T_A\leq 2n)\nonumber\\
& \leq 2n\mu(A)((\psi(n)+1)(\mu(T_A>f_A-2n))^k. \label{pr26}
\end{align}

For the case $\phi$, we apply Lemma \ref{lemfat}-(b) and get
\begin{align}
\eqref{pr23}&\le (\mu(T_A>f_A-2n)+\phi(n))^{k}(2n\mu(A)+\phi(n)).\label{pr22}
\end{align}

By $\psi$-mixing and \eqref{lf2}
\begin{align}
\eqref{pr24}&\leq\mu({T_A}>kf_A-2n)\psi(n)\nonumber\\
\label{pr27} &\leq (\psi(n)+1)^{k-1}\mu(T_A>f_A-2n)^k\psi(n)
\end{align} 

And applying $\phi$-mixing and \eqref{lf3}
\begin{align}
\label{pr21}\eqref{pr24}&\leq(\mu({T_A}>f_A-2n)+\phi(n))^k\phi(n)
\end{align} 

Finally, using shift-invariance and the same arguments as above
\begin{align}
 \eqref{pr25} 
&=\mu({T_A}>kf_A)\mu\left(r-2n<{T_A}\leq r\right)\nonumber\\
\label{pr28} &\leq  2n\mu(A)\mu({T_A}>kf_A-2n).
\end{align}

Therefore, \eqref{lf2}, \eqref{pr26},\eqref{pr27} and \eqref{pr28} give us
\begin{align*}
\eqref{pr23}+\eqref{pr24}+\eqref{pr25}
\leq 2(M+1)(\psi(n)+1)^{k-1}\mu(T_A>f_A-2n)^k\epsilon_{\psi}(A)
\end{align*}
and from \eqref{lf3}, \eqref{pr22}, \eqref{pr21} and \eqref{pr28} we get
\begin{align*}
\eqref{pr23}+\eqref{pr24}+\eqref{pr25}\leq 4(\mu(T_A>f_A-2n)+\phi(n))^k\epsilon_{\phi}(A)
\end{align*}
which ends the proof for (a)-1 and (a)-2. 

For the proof of the (a)-3, we write a similar triangle inequality as above:
\begin{align*}
& |\mu_A({T_A}>kf_A+r)-\mu_A({T_A}>kf_A)\mu({T_A}>r)|\nonumber \\
&\leq  \left|\mu_A\left({T_A}>kf_A;{T_A}^{[kf_A]}>r\right)-\mu_A\left({T_A}>kf_A;{T_A}^{[kf_A+2n]}>r-2n\right)\right|\\
&+\left|\mu_A\left({T_A}>kf_A;{T_A}^{[kf_A+2n]}>r-2n\right)-\mu_A\left({T_A}>kf_A\right)\mu\left({T_A}^{[kf_A+2n]}>r-2n\right)\right|\\
&+\mu_A\left({T_A}>kf_A\right)\left|\mu\left({T_A}^{[kf_A+2n]}>r-2n\right)-\mu\left({T_A}>r\right)\right|.
\end{align*}

Then, we follow the same as we did for (a)-1, but applying item (c) of Lemma \ref{lemfat} and using the $\psi$-mixing property:
\[
|\mu_A(B;C)-\mu_A(B)\mu(C)|\leq \mu_A(B)\mu(C)\psi(n)
\]
where $A,B\in\mathcal{F}_0^i$ and $C\in\mathcal{F}_{i+n}^{\infty}$. For the case $r<2n$, we use 
\begin{align*}
&|\mu_A({T_A}>kf_A+r)-\mu_A({T_A}>kf_A)\mu({T_A}>r)|\\
&\leq  \left|\mu_A\left({T_A}>kf_A,{T_A}^{[kf_A]}>r\right)-\mu_A({T_A}>kf_A)\right|+\mu_A({T_A}>kf_A)|1-\mu({T_A}>r)|
\end{align*}
and proceed as we did in \eqref{pr29}, applying again Lemma \ref{lemfat}-(c). This ends item (a). 

We now come to the proof of items (b)-1 and (b)-2.  For $k = 1$ we have an equality. For $k\geq 2$ we get
\begin{align}
& \left|\mu({T_A}>kf_A)-\mu({T_A}>f_A)^{k}\right|\nonumber\\
&=\left|\sum_{j=2}^k\left(\mu({T_A}>jf_A)-\mu({T_A}>(j-1)f_A)\mu({T_A}>f_A)\right)\mu({T_A}>f_A)^{k-j}\right|\nonumber\\
\label{pr35}&\leq \sum_{j=2}^k \left|\mu({T_A}>jf_A)-\mu({T_A}>(j-1)f_A)\mu({T_A}>f_A)\right|\mu({T_A}>f_A)^{k-j}.
\end{align}

We put $r=f_A$ in item (a)-1 to obtain (b)-1:
\begin{align*}
\eqref{pr35}&\leq2(M+1)\epsilon_{\psi}(A)\sum_{j=2}^k (\psi(n)+1)^{j-2}\mu({T_A}>f_A-2n)^{j-1}\mu({T_A}>f_A)^{k-j}\\
&\leq 2(M+1)\epsilon_{\psi}(A)(k-1)(\psi(n)+1)^{k-2}\mu({T_A}>f_A-2n)^{k-1}.
\end{align*}
 
Furthermore, we get the inequality (b)-2, under $\phi$-mixing, proceeding similarly as above
\begin{align*}
\eqref{pr35}&\leq 4\epsilon_{\phi}(A)\sum_{j=2}^k (\mu({T_A}>f_A-2n)+\phi(n))^{j-1}(\mu({T_A}>f_A-2n)+\phi(n)^{k-j}\\
&= 4\epsilon_{\phi}(A)(k-1)(\mu({T_A}>f_A-2n)+\phi(n))^{k-1}
\end{align*}

Finally, we proof (b)-3 applying (a)-3 as follows
\begin{align*}
& \left|\mu_A({T_A}>kf_A)-\mu_A({T_A}>f_A)\mu(T_A>f_A)^{k-1}\right|\\
&\leq\sum_{j=2}^k\left|\mu_A({T_A}>jf_A)-\mu_A({T_A}>(j-1)f_A)\mu({T_A}>f_A)\right|\mu({T_A}>f_A)^{k-j}\\
&\leq 2(M+1)\epsilon_{\psi}(A)\sum_{j=2}^{k}((\psi(n)+1)\mu(T_A>f_A-2n))^{j-2}\mu(T_A>f_A)^{k-j}\\
& \leq 2(M+1)\epsilon_{\psi}(A)(k-1)((\psi(n)+1)\mu(T_A>f_A-2n))^{k-2}.
\end{align*}


\end{proof}

The next two lemmas are classical results and are stated without proof. The first one establishes the reversibility of certain sets for stationary measures and the second one is a discrete version of the Mean Value Theorem which follows with a straightforward computation.

\begin{lemma}\label{Est}
Let $\mu$ be shift-invariant. For all positive $i\in\mathbb{N}$, $n\geq 1$ and $A\in\mathcal{A}^n$ we have
$$\mu({T_A}=i)=\mu({T_A}>i-1;A)$$
\end{lemma}



\begin{lemma}\label{produt}
Given $a_1,...,a_n,b_1,...,b_n$ real numbers such that $0\leq a_i,b_i\leq 1$, the following inequality holds
\begin{align*}
\left|\prod_{i=1}^n a_i-\prod_{i=1}^n b_i\right| \leq \sum_{i=1}^n\left|a_i-b_i\right|\left(\max_{1\leq i\leq n}\{a_i,b_i\}\right)^{n-1}\leq \sum_{i=1}^n\left|a_i-b_i\right|.
\end{align*}
\end{lemma}

\subsection{Proof of Theorem \ref{theo:correct}}\label{subsectheo1}

Theorem \ref{theo:correct} contains 8 statements, each statement corresponding to a choice of 
\begin{itemize}
\item recurrence time: hitting or return,
\item mixing property: $\psi$ or $\phi$,
\item amplitude of $t$:  smaller or larger than $f_A$.\\
\end{itemize}

{Recall the definition of $n'$ in \eqref{defn'}. The proof of Theorem \ref{theo:correct} holds for all $n\geq n_0$, where $n_0$ is explicitly given by
\begin{align}
\label{defn0} n_0:=\inf\left\lbrace m\geq n'; \sup_{A\in\mathcal{A}^n} \mu(A)\tau(A) < 1/2,\;\forall n\geq m\right\rbrace
\end{align}
which is finite since $\sup_{A\in \mathcal A^n} \mu(A)\tau(A)\stackrel{n}\longrightarrow 0$. Then, in particular, we have $\tau(A)<f_A$ for all $n\geq n_0$ and $A\in\mathcal{A}^n$. 
}
 

\subsubsection{Proofs of the statements for small $t$'s} 

Here we assume that $1\leq t\leq f_A:=[2\mu(A)]^{-1}$.

\begin{proof}[Proof of hitting time, $\phi$ and $\psi$ together]
Recall that $\epsilon(A)$ denotes $\epsilon_{\phi}(A)$ or $\epsilon_{\psi}(A)$, depending on whether the measure is $\phi$ or $\psi$-mixing. For positive $i\in\mathbb{N}$, define 
\[
p_i=\frac{\mu_A({T_A}>i-1)}{\mu({T_A}>i-1)}.
\]
Then
\begin{align}
\mu({T_A}>t)&=\prod_{i=1}^t \frac{\mu({T_A}>i)}{\mu({T_A}>i-1)}= \prod_{i=1}^t \left(1-\mu({T_A}=i|{T_A}>i-1)\right)\nonumber\\
&\label{teo4}=\prod_{i=1}^t \left(1-\mu(\sigma^{-i}(A)|{T_A}>i-1)\right)=\prod_{i=1}^t \left(1-\mu(A)p_i\right)
\end{align}
where we used Lemma \ref{Est} in the last equality.

Similarly, for $\tau(A)\leq t\leq f_A$, we have
\begin{align}
\label{teo15} \mu({T_A}>t)=\mu({T_A}>\tau(A))\prod_{i=\tau(A)+1}^t \left(1-\mu(A)p_i\right).
\end{align}

We apply \eqref{teo4} and Lemma \ref{produt} to obtain
\begin{align}
&\quad\left|\mu({T_A}>t)-e^{-\rho(A)\mu(A)t}\right| \nonumber\\
&=\left|\prod_{i=1}^t \left(1-\mu(A)p_i\right)-\prod_{i=1}^t e^{-\rho(A)\mu(A)}\right|\nonumber\\
&\leq \left|\prod_{i=1}^{\tau(A)}\left(1-\mu(A)p_i\right)-\prod_{i=1}^{\tau(A)} e^{-\rho(A)\mu(A)}\right|+\left|\prod_{i=\tau(A)+1}^t \left(1-\mu(A)p_i\right)-\prod_{i=\tau(A)+1}^t e^{-\rho(A)\mu(A)}\right|\nonumber\\
\label{teo2}&\leq  \left|\mu({T_A}>\tau(A))-e^{-\rho(A)\mu(A)\tau(A)}\right|+\sum_{i=\tau(A)+1}^t \left|1-\mu(A)p_i-e^{-\rho(A)\mu(A)}\right|.
\end{align}

Applying the inequality $\left|1-e^{-x}\right|\leq x$ for $x\geq 0$, we have
\begin{align}
\left|\mu({T_A}>\tau(A))-e^{-\rho(A)\mu(A)\tau(A)}\right|&\leq \left|\mu({T_A}>\tau(A))-1\right|+\left|1-e^{-\rho(A)\mu(A)\tau(A)}\right|\nonumber\\
&\label{teo16}\leq 2\tau(A)\mu(A).
\end{align}

On the other hand, by the triangle inequality
\begin{align}
\label{teo1}\left|1-p_i\mu(A)-e^{-\rho(A)\mu(A)}\right|&\leq \left|p_i-\rho(A)\right|\mu(A)+\left|1-\rho(A)\mu(A)-e^{-\rho(A)\mu(A)}\right|.
\end{align}

Since $|1-x-e^{-x}|\leq \frac{x^2}{2}$ for all $0\leq x\leq 1$, by doing $x=\rho(A)\mu(A)$ we get
\begin{align*}
\left|1-\rho(A)\mu(A)-e^{-\rho(A)\mu(A)}\right|\leq \frac{\rho(A)^2\mu(A)^2}{2}\leq \frac{\epsilon(A)\mu(A)}{2}.
\end{align*}

Furthermore, for $\tau(A)+1\leq i\leq f_A+1$, Proposition \ref{PR} gives us
\begin{align*}
|p_i-\rho(A)|=\left|\frac{\mu_A({T_A}>i-1)}{\mu({T_A}>i-1)}-\rho(A)\right|\leq \frac{C\epsilon(A)}{\mu({T_A}>i-1)}\leq 2C\epsilon(A),
\end{align*}
where, for the last inequality we used
\begin{align*}
\mu({T_A}>i-1)=1-\mu({T_A}\leq i-1)\geq 1-(i-1)\mu(A)\geq 1-f_A\mu(A)=\frac{1}{2}.
\end{align*}


Thus, applying \eqref{teo1} we obtain for $\tau(A)+1\leq i\leq f_A+1$
\begin{align}
\left|1-p_i\mu(A)-e^{-\rho(A)\mu(A)}\right|
\label{teo5}\leq\left(2C+1/2\right)\epsilon(A)\mu(A).
\end{align}

Therefore, \eqref{teo2}, \eqref{teo16} and \eqref{teo5} give us
\begin{align}
 \left|\mu({T_A}>t)-e^{-\rho(A)\mu(A)t}\right|&\leq 2\tau(A)\mu(A)+\left(2C+1/2\right)(t-\tau(A))\epsilon(A)\mu(A)\nonumber\\
\label{teo10}&\le \left(2C+1/2\right)[\tau(A)\mu(A)+t\mu(A)\epsilon(A)]
\end{align}
which concludes the statement of Theorem \ref{theo:correct} for hitting time at small $t$'s (with either $\phi$ or $\psi$).\\
\end{proof}


\begin{proof}[Proof for return time, $\phi$ and $\psi$ together]
By definition we have $\mu_A(T_A>t)=p_{t+1}\mu(T_A>t)$. Then, we use again the triangle inequality to write
\begin{align}
&\quad\left|\mu_A({T_A}>t)-\rho(A)e^{-\rho(A)\mu(A)(t-\tau(A))}\right|\nonumber\\
\label{teo3} &\leq  \mu({T_A}>t)\left|p_{t+1}-\rho(A)\right|+\rho(A)\left|\mu({T_A}>t)-e^{-\rho(A)\mu(A)(t-\tau(A))}\right|.
\end{align}

As we saw before, the first modulus above is bounded by $2C\epsilon(A)$. On the other hand, applying \eqref{teo15} we can write
\begin{align*}
\left|\mu({T_A}>t)-e^{-\rho(A)\mu(A)(t-\tau(A))}\right|=\left|\mu({T_A}>\tau(A))\prod_{i=\tau(A)+1}^t(1-\mu(A)p_i)-\prod_{i=\tau(A)+1}^t e^{-\rho(A)\mu(A)}\right|. 
\end{align*}
This is bounded, applying Lemma \ref{produt}, by
\begin{align*}
&\quad\; |\mu({T_A}>\tau(A))-1|+\left|\prod_{i=\tau(A)+1}^t(1-\mu(A)p_i)-\prod_{i=\tau(A)+1}^t e^{-\rho(A)\mu(A)}\right|\nonumber\\
&\leq \tau(A)\mu(A)+\left(2C+1/2\right)t\mu(A)\epsilon(A)
\end{align*}
where the last inequality follows from \eqref{teo2} and \eqref{teo5}. Finally, notice that $t\mu(A)\leq f_A\mu(A)=1/2$ and $\tau(A)\mu(A)\leq 2\epsilon(A)$ (use Lemma \ref{lemtau} for $\psi$). Therefore, we obtain from \eqref{teo3}
\begin{align}
\label{teo19}\left|\mu_A({T_A}>t)-\rho(A)e^{-\rho(A)\mu(A)(t-\tau(A))}\right|\leq \left(3C+9/4\right)\epsilon(A)
\end{align}
This concludes the statement of Theorem \ref{theo:correct} for return time at small $t$'s (with either $\phi$ or $\psi$).
 \end{proof}

\subsubsection{Proof of the statements for large $t$'s}

The proof for return time for $t>f_A$ is done in \cite{abadi/vergne/2009} under $\phi$-mixing, finite alphabet and complete grammar. 
The proof still holds if one just assume  countable alphabet and incomplete grammar (recall Remark \ref{rem.error} for the uniform convergence to zero of the error term $\epsilon_{\phi}$). 
 Thus, we focus on  hitting time under each mixing assumption, and return time only under $\psi$-mixing.



\begin{proof}[Proof of Theorem \ref{theo:correct} for hitting times, for $t>f_A$] Write $t=kf_A+r$ with integer $k\ge1$ and $0\leq r<f_A$. Thus, we have
\begin{align}
\left|\mu({T_A}>t)-e^{-\rho(A)\mu(A)t}\right|&\leq  \left|\mu({T_A}>kf_A+r)-\mu({T_A}>kf_A)\mu({T_A}>r)\right|\label{teo6}\\
\label{teo7}&\quad+ \left|\mu({T_A}>kf_A)-\mu({T_A}>f_A)^{k}\right|\mu({T_A}>r)\\
\label{teo8}&\quad+ \left|\mu({T_A}>f_A)^{k}-e^{-\rho(A)\frac{k}{2}}\right|\mu({T_A}>r)\\
\label{teo9}&\quad+ \left|e^{-\rho(A)\frac{k}{2}}\mu({T_A}>r)-e^{-\rho(A)\mu(A)t}\right|.
\end{align}

In order to get an upper bound for the sum of (\ref{teo6}) and (\ref{teo7}), we  analyse the  $\psi$ and $\phi$ cases separately, and start by the $\psi$-mixing.  Applying items (a)-1 and (b)-1 of Proposition \ref{prfat}, that sum is bounded by
\begin{align}
&\leq C'\epsilon_{\psi}(A)(\psi(n)+1)^{k-1}\mu({T_A}>f_A-2n)^{k}\left(1+(k-1)((\psi(n)+1)\mu({T_A}>f_A-2n))^{-1}\right)\nonumber\\
&\leq 2(M+1)\epsilon_{\psi}(A)\left((\psi(n)+1)\mu({T_A}>f_A-2n)\right)^{k}2k\nonumber\\
\label{teo12}&\leq 8(M+1)\epsilon_{\psi}(A)\mu(A)t\left((\psi(n)+1)\mu({T_A}>f_A-2n)\right)^{k}.
\end{align}
where the last two inequalities are justified by $\mu({T_A}>f_A-2n)^{-1}\leq \mu({T_A}>f_A)^{-1}\leq 2$ and $k\leq 2\mu(A)t$.


On the other hand, applying (\ref{teo10}) with $t=f_A-2n$ we get 
\begin{align*} 
\left|\mu({T_A}>f_A-2n)-e^{-\rho(A)\mu(A)(f_A-2n)}\right|&=\left|\mu({T_A}>f_A-2n)-e^{-\frac{\rho(A)}{2}+2n\rho(A)\mu(A)}\right|\nonumber\\
&\leq  \left(2C+1/2\right)(\tau(A)\mu(A)+(f_A-2n)\mu(A)\epsilon(A))\nonumber\\
&\leq \left(5C+5/4\right)\epsilon(A)
\end{align*}
where we use $\tau(A)\mu(A)\leq 2\epsilon(A)$. 

Furthermore, by the Mean Value Theorem (MVT)
\begin{align*}
\left|e^{-\frac{\rho(A)}{2}+2n\rho(A)\mu(A)}-e^{-\frac{\rho(A)}{2}}\right|&\leq 2n\rho(A)\mu(A)e^{-\frac{\rho(A)}{2}+2n\rho(A)\mu(A)}\nonumber\\
&\leq  2n\mu(A)e^{2n\mu(A)}\leq \frac{11}{2} n\mu(A)
\end{align*}
since for $n\geq n_0$ we have $2n\mu(A)\leq 2\sup \mu(A)\tau(A)\leq 1$.

Thus, it follows that
\begin{align*}
&\quad \left|(\psi(n)+1)\mu({T_A}>f_A-2n)-e^{-\frac{\rho(A)}{2}}\right|\\
&\leq  \psi(n)+\left|\mu({T_A}>f_A-2n)-e^{-\frac{\rho(A)}{2}+2n\rho(A)\mu(A)}\right|+\left|e^{-\frac{\rho(A)}{2}+2n\rho(A)\mu(A)}-e^{-\frac{\rho(A)}{2}}\right|\\
&\leq \left(5C+27/4\right)\epsilon(A).
\end{align*}
Therefore
\[
\left((\psi(n)+1)\mu({T_A}>f_A-2n)\right)^k\leq \left(e^{-\frac{\rho(A)}{2}}+\left(5C+27/4\right)\epsilon(A)\right)^k.
\]
Since $e^x-1\geq x\;\forall x\in\mathbb{R}$, by doing $K=\left(5C+27/4\right)e^{1/2}$ we get
\begin{align}
&\left(e^{K\epsilon(A)}-1\right)\geq K\epsilon(A)\geq \left(5C+27/4\right)\epsilon(A) e^{\frac{\rho(A)}{2}}\nonumber\\
\Longrightarrow & \;e^{-\frac{\rho(A)}{2}}\left(e^{K\epsilon(A)}-1\right)\geq \left(5C+27/4\right)\epsilon(A)\nonumber\\
\label{teo11}\Longrightarrow &  \;e^{-\frac{\rho(A)}{2}+K\epsilon(A)}\geq \left(5C+27/4\right)\epsilon(A)+e^{-\frac{\rho(A)}{2}}.
\end{align}
Now, using that $k=2\mu(A)(t-r)$, we have
\begin{align}
\left((\psi(n)+1)\mu({T_A}>f_A-2n)\right)^{k} & \leq \left(e^{-\frac{\rho(A)}{2}+K\epsilon(A)}\right)^{k}\nonumber\\
& =e^{-\rho(A)\mu(A)t+\rho(A)\mu(A)r+2K\epsilon(A)\mu(A)t-2K\epsilon(A)\mu(A)r}\nonumber\\
& \leq e^{-\mu(A)t\left(\rho(A)-2K\epsilon(A)\right)}e^{\mu(A)r}\nonumber\\
\label{teo21} &\leq e^{1/2}e^{-\mu(A)t\left(\rho(A)-C_3\epsilon(A)\right)}
\end{align}
where the last inequality follows from $e^{\mu(A)r}\leq e^{\mu(A)f_A}$.

Therefore, it follows from (\ref{teo12}) that the sum of (\ref{teo6}) and (\ref{teo7}) is bounded by
\begin{align*}
14(M+1)\epsilon_{\psi}(A)\mu(A)te^{-\mu(A)t\left(\rho(A)-C_3\epsilon(A)\right)}.
\end{align*}
 

We now turn to the case of $\phi$-mixing. We apply items (a)-2 and (b)-2 of Proposition \ref{prfat} to get an upper bound for the sum of (\ref{teo6}) and (\ref{teo7}):
\begin{align*}
&\left|\mu({T_A}>kf_A+r)-\mu({T_A}>kf_A)\mu({T_A}>r)\right|+\left|\mu({T_A}>kf_A)-\mu({T_A}>f_A)^k\right|\mu({T_A}>r)\nonumber\\
&\leq 4\epsilon_{\phi}(A)\left(\mu({T_A}>f_A-2n)+\phi(n)\right)^{k}\left(1+(k-1)\left(\mu({T_A}>f_A-2n)+\phi(n)\right)^{-1}\right)\nonumber\\
&\leq 4\epsilon_{\phi}(A)\left(\mu({T_A}>f_A-2n)+\phi(n)\right)^{k}2k\nonumber\\
&\leq 16\epsilon_{\phi}(A)\mu(A)t\left(\mu({T_A}>f_A-2n)+\phi(n)\right)^{k}.
\end{align*}


Similarly to $\psi$-mixing case, one obtain
\begin{align*}
\left(\mu({T_A}>f_A-2n)+\phi(n)\right)^{k}\leq e^{1/2}e^{-\mu(A)t\left(\rho(A)-C_3\epsilon(A)\right)}
\end{align*}
which implies in the $\phi$-mixing case that the sum of (\ref{teo6}) and (\ref{teo7}) is bounded by
$$27\epsilon_{\phi}(A)\mu(A)te^{-\mu(A)t\left(\rho(A)-C_3\epsilon(A)\right)}.$$


Now, we will treat the cases $\psi$ and $\phi$ together to obtain upper bounds for (\ref{teo8}) and (\ref{teo9}). In order to get an upper bound for (\ref{teo8}), we apply (\ref{teo10}) with $t=f_A$:
\begin{align}
\left|\mu({T_A}>f_A)-e^{-\rho(A)\mu(A)f_A}\right|&=\left|\mu({T_A}>f_A)-e^{-\frac{\rho(A)}{2}}\right|\nonumber\\
&\leq  \left(2C+1/2\right)\left(\tau(A)\mu(A)+f_A\mu(A)\epsilon(A)\right)\nonumber\\
\label{teo22}&\leq  \left(5C+5/4\right)\epsilon(A).
\end{align}

Thus, applying Lemma \ref{produt} we have
\begin{align*}
&\quad\left|\mu({T_A}>f_A)^{k}-e^{-\rho(A)\frac{k}{2}}\right|\\
&\leq  \sum_{i=1}^k\left|\mu({T_A}>f_A)-e^{-\frac{\rho(A)}{2}}\right|\left(\max\left\lbrace\mu({T_A}>f_A),e^{-\frac{\rho(A)}{2}}\right\rbrace\right)^{k-1}.\nonumber
\end{align*}

The max is bounded using \eqref{teo22} by
%
\begin{align*}
e^{-\frac{\rho(A)}{2}}+\left(5C+5/4\right)\epsilon(A).
\end{align*}
Naturally, the absolute value is also bounded using \eqref{teo22} and we get that the above sum is bounded above by
\begin{align}
\label{teo20} k\,\left(5C+5/4\right)\epsilon(A)\,\left(e^{-\frac{\rho(A)}{2}}+\left(5C+5/4\right)\epsilon(A)\right)^{k-1}.
\end{align}
Recalling that $k=2\mu(A)(t-r)$ and proceeding as we did for \eqref{teo11} and \eqref{teo21}, we get the following upper bound for  \eqref{teo8}
\[
2\left(5C+5/4\right)\epsilon(A)\mu(A)t\,e^{-\mu(A)t\left(\rho(A)-C_3\epsilon(A)\right)}e^{1}\leq
7 \left(4C+1\right)\epsilon(A)\mu(A)te^{-\mu(A)t\left(\rho(A)-C_3\epsilon(A)\right)}.
\]

To conclude the proof for hitting time, we apply (\ref{teo10}) with $t=r$ to bound (\ref{teo9}) as follows
\begin{align*}
\left|e^{-\rho(A)\frac{k}{2}}\mu({T_A}>r)-e^{-\rho(A)\mu(A)t}\right|&
=  e^{-\rho(A)\mu(A)t+\rho(A)\mu(A)r}\left|\mu({T_A}>r)-e^{-\rho(A)\mu(A)r}\right|\nonumber\\
&\leq  \left(2C+1/2\right) e^{-\rho(A)\mu(A)t+\mu(A)f_A}\left(\tau(A)\mu(A)+r\mu(A)\epsilon(A)\right)\nonumber\\
&\leq  \left(2C+1/2\right)\left(\tau(A)\mu(A)+f_A\mu(A)\epsilon(A)\right)e^{-\rho(A)\mu(A)t}e^{1/2}\nonumber\\
&\leq  (17C+5)\epsilon(A)\mu(A)te^{-\mu(A)t\left(\rho(A)-C_3\epsilon(A)\right)}
\end{align*}
where the term $\mu(A)t$ follows from $1= 2\mu(A)f_A\leq 2\mu(A)t$. 
\end{proof}

\begin{proof}[Proof of Theorem \ref{theo:correct} for return time, for $t>f_A$ and under $\psi$-mixing] We use again the triangle inequality to write
\begin{align}
&\left|\mu_A({T_A}>t)-\rho(A)e^{-\rho(A)\mu(A)(t-\tau(A))}\right|\nonumber\\
\label{teo13}&\leq \left|\mu_A({T_A}>kf_A+r)-\mu_A({T_A}>kf_A)\mu({T_A}>r)\right|\\
\label{teo14}&+ \left|\mu_A({T_A}>kf_A)-\mu_A(T_A>f_A)\mu({T_A}>f_A)^{k-1}\right|\mu({T_A}>r)\\
\label{teo17}&+ \left|\mu_A(T_A>f_A)\mu({T_A}>f_A)^{k-1}-\rho(A)e^{-\rho(A)\frac{k}{2}}\right|\mu({T_A}>r)\\
\label{teo18}&+ \rho(A)e^{-\rho(A)\frac{k}{2}}\left|\mu({T_A}>r)-e^{-\rho(A)\mu(A)(r-\tau(A))}\right|.
\end{align}
 
Applying items (a)-3 and (b)-3 of Proposition \ref{prfat}, the sum of \eqref{teo13} and \eqref{teo14} is bounded by
\begin{align*}
& 2(M+1)\epsilon_{\psi}(A)((\psi(n)+1)\mu(T_A>f_A-2n))^{k-1}(1+(k-1)((\psi(n)+1)\mu(T_A>f_A-2n))^{-1})\\
& \leq 2(M+1)\epsilon_{\psi}(A)((\psi(n)+1)\mu(T_A>f_A-2n))^{k-1}2k\\
& \leq 8(M+1)\epsilon_{\psi}(A)\mu(A)t((\psi(n)+1)\mu(T_A>f_A-2n))^{k-1}.
\end{align*}

Replacing $k$ by $k-1$ in \eqref{teo21}, the last term is bounded above by
\[
8(M+1)\epsilon_{\psi}(A)\mu(A)t\,e^{1}\,e^{-\mu(A)t\left(\rho(A)-C_3\epsilon(A)\right)}\leq 22(M+1)\epsilon_{\psi}(A)\mu(A)t\,e^{-\mu(A)t\left(\rho(A)-C_3\epsilon(A)\right)}.
\]

On the other hand, Lemma \ref{produt} gives us
\begin{align*}
\eqref{teo17} & \leq \left(\max\left\lbrace\mu_A(T_A>f_A),\mu(T_A>f_A),e^{-\rho(A)/2}\right\rbrace\right)^{k-1}\left(\left|\mu_A(T_A>f_A)-\rho(A)e^{-\rho(A)/2}\right|\right.\\
&\left.+\sum_{i=1}^{k-1}\left|\mu(T_A>f_A)-e^{-\rho(A)/2}\right|\right)
\end{align*}

The last sum is bounded by $(5C+5/4)(k-1)\epsilon(A)$ using \eqref{teo22}. On the other hand, applying \eqref{teo19} with $t=f_A$ and the MVT we obtain
\begin{align*}
& \left|\mu_A(T_A>f_A)-\rho(A)e^{-\rho(A)/2}\right|\\
& \leq \left|\mu_A(T_A>f_A)-\rho(A)e^{-\rho(A)/2+\rho(A)\mu(A)\tau(A)}\right|+\rho(A)\left|e^{-\rho(A)/2+\rho(A)\mu(A)\tau(A)}-e^{-\rho(A)/2}\right|\\
& \leq (3C+9/4)\epsilon(A)+\rho(A)\mu(A)\tau(A)e^{-\rho(A)(1/2-\mu(A)\tau(A))}\\
& \leq (3C+17/4)\epsilon(A)
\end{align*}
since $e^{-\rho(A)(1/2-\mu(A)\tau(A))}\leq 1$ and $\mu(A)\tau(A)\leq 2\epsilon_{\psi}(A)$ for $n\geq n_0$.

Furthermore, the last inequality implies
\[
\mu_A(T_A>f_A)\leq \rho(A)e^{-\rho(A)/2}+(3C+17/4)\epsilon(A)\leq e^{-\rho(A)/2}+(5C+5/4)\epsilon(A)
\]
and by \eqref{teo22} we get
\[
\max\left\lbrace\mu_A(T_A>f_A),\mu(T_A>f_A),e^{-\rho(A)/2}\right\rbrace\leq e^{-\rho(A)/2}+(5C+5/4)\epsilon(A).
\]

Therefore, as we saw in \eqref{teo20}, we have
\begin{align*}
\eqref{teo17}& \leq (5C+5/4)\epsilon(A)k\left(e^{-\rho(A)/2}+(5C+5/4)\epsilon(A)\right)^{k-1}\\
& \leq 2(5C+5/4)\epsilon(A)\mu(A)t\,e^{1}\,e^{-\mu(A)t\left(\rho(A)-C_3\epsilon(A)\right)}\\
&\leq (109M+116)\epsilon(A)\mu(A)te^{-\mu(A)t\left(\rho(A)-C_3\epsilon(A)\right)}.
\end{align*}

Finally, by doing $t=r$ in \eqref{teo10} and applying the MVT once again, we get
\begin{align*}
& \left|\mu({T_A}>r)-e^{-\rho(A)\mu(A)(r-\tau(A))}\right|\\
&\leq \left|\mu(\tau_A>r)-e^{-\rho(A)\mu(A)r}\right|+\left|e^{-\rho(A)\mu(A)r}-e^{-\rho(A)\mu(A)(r-\tau(A))}\right|\\
&\leq (2C+1/2)(\tau(A)\mu(A)+r\mu(A)\epsilon(A))+\rho(A)\mu(A)\tau(A)e^{-\rho(A)\mu(A)(r-\tau(A))}\\
&\leq (2C+1/2)(2\epsilon(A)+f_A\mu(A)\epsilon(A))+(7/2)\epsilon(A)\\
& \leq (5C+19/4)\epsilon(A).
\end{align*}

We justify the third inequality in two cases. If $r>\tau(A)$, then $e^{-\rho(A)\mu(A)(r-\tau(A))}\leq 1$. Otherwise, if $r\leq \tau(A)$, then $e^{-\rho(A)\mu(A)(r-\tau(A))}\leq e^{\rho(A)\mu(A)\tau(A)}\leq e^{1/2}$, since $n\geq n_0$. Now just note that $\rho(A)\mu(A)\tau(A)\leq 2\epsilon(A)$.

Therefore, we finish the proof obtaining the following upper bound:
\begin{align*}
\eqref{teo18} & \leq (5C+19/4)\epsilon(A)e^{\rho(A)\mu(A)r}e^{-\rho(A)\mu(A)t}\\
& \leq (5C+19/4)\epsilon(A)2\mu(A)t\;e^{f_A\mu(A)}e^{-\mu(A)t(\rho(A)-C_3\epsilon(A))}\\
& \leq (66M+82)\epsilon(A)\mu(A)te^{-\mu(A)t(\rho(A)-C_3\epsilon(A))}.
\end{align*}
\end{proof}

 \subsection{Proof of Theorem \ref{theo:positive}}
\begin{proof}[Proof of Statement (a)] For each $x\in\mathcal{X}$ we define 
\[
\tau(x):=\sup\{\tau(x_0^{n-1}),n\geq 1\}.
\]
Let $\mathcal{B}=\{x\in\mathcal{X};\tau(x)=\infty\}$ be the set of aperiodic points of $\mathcal{X}$. For $x\in\mathcal{B}$, denote $A_n=A_n(x)$ and consider the case $\tau(A_n)<n$. Then, we have
\begin{align*}
1-\rho(A_n)&=\mu_{A_n}\left(T_{A_n}=\tau(A_n)\right)\\
& = \mu_{A_n}\left(\sigma^{-n}\left(A_n^{(\tau(A_n)}\right)\right)\\
& \leq \mu_{A_n}\left(\sigma^{-n-\lfloor\tau(A_n)/2\rfloor}\left(A_n^{(\lceil\tau(A_n)/2\rceil}\right)\right)\\
& \leq \mu\left(A_n^{(\lceil\tau(A_n)/2\rceil}\right)+\phi\left(\lfloor\tau(A_n)/2\rfloor+1\right).
\end{align*}
Since $x\in\mathcal{B}$, we have $\tau(A_n)\stackrel{n}\longrightarrow\infty$, which implies that the last expression converges to zero.
For the case $\tau(A_n)\geq n$, we use the same argument
\begin{align*}
1-\rho(A_n) & =\mu_{A_n}\left(\sigma^{-\tau(A_n)}(A_n)\right)\\
& \leq \mu_{A_n}\left(\sigma^{-\tau(A_n)-\lfloor n/2\rfloor}\left(A_n^{(\lceil n/2\rceil}\right)\right)\\
& \leq \mu\left(A_n^{(\lceil n/2\rceil)}\right)+\phi\left(\lfloor n/2\rfloor+1\right)
\end{align*}
which also converges to zero. Therefore, $\rho(A_n)\stackrel{n}\longrightarrow 1$. We conclude the proof by noting that $\mathcal{X}-\mathcal{B}$ is a countable set, and thus $\mu(\mathcal{B})=1$.
\end{proof}

\begin{proof}[Proof of Statement (b)] By Lemma \ref{lemtau}, for $\psi$-mixing or summable $\phi$-mixing measures, there exists $n_0\geq 1$ such that
\[
\forall n\geq n_0,\;\forall A\in\mathcal C_n,\;\;\mu(A)^{-1}> \tau(A)\,.
\]
Now, since $\mu_{A}(T_{A} > j),j\geq 1$ is  a nonincreasing sequence, the potential well is larger or equal than
the arithmetic mean of the subsequent  $\mu(A)^{-1}$ elements 
\begin{align}
\nonumber
\rho(A)=\mu_{A}(T_{A} > \tau(A) )
&\ge \frac{1}{\mu(A)^{-1} -\tau(A)}   \sum_{j=\tau(A)}^{\mu(A)^{-1}-1} \mu_{A}(T_{A} > j)\\
\nonumber
&\geq \frac{1}{\mu(A)^{-1}}   \sum_{j=\tau(A)}^{\mu(A)^{-1}-1} \mu_{A}(T_{A} > j)\\
& = \sum_{j=\tau(A)}^{\mu(A)^{-1}-1} \mu(A; T_{A} > j)\nonumber\\
&= \sum_{j=\tau(A)}^{\mu(A)^{-1}-1} \mu(T_{A} = j+1).
\label{chou}
\end{align}
In the last equality we used  Lemma \ref{Est}.
By \eqref{chou} one obtain
\begin{align}\label{chichi}
\rho(A)
&\geq  \mu(    T_{A}  \le \mu(A)^{-1}  )  - \mu(T_{A}  \le \tau(A))    \nonumber\\
&=   
\mu(T_{A}  \le \mu(A)^{-1} )    -  \tau(A)\mu(A)
\end{align}
where the equality follows by stationarity and the definition of $\tau(A)$.

By Lemmas \ref{lemsubexp} and \ref{lemtau}, we know that $\tau(A)\mu(A)\stackrel{n}\longrightarrow 0$ uniformly. Thus, it is enough  to find  a strictly positive lower bound for $\mu(T_{A}  \le \mu(A)^{-1} )$. Let
\[
N=\sum_{j=1}^{\mu(A)^{-1}} \un_{A}\circ \sigma^{j}\,
\]
which counts the number of occurrences of $A$ up to $\mu(A)^{-1}$. By the so-called \emph{second moment method}, 
\begin{equation} \label{chouchouchou}
 \mu(T_{A}  \le  \mu(A)^{-1} ) =\mu(N\ge1)\ge \frac{\E(N)^2}{\E(N^2)} .
\end{equation}
Stationarity gives $\E(N)=1$. It remains to
prove that 
$\E(N^2)$ is bounded above by a constant. Expanding $N^2$, using stationarity and $\E(N)=1$ we obtain
\begin{equation} \label{geral}
\E(N^2)
=1+ 2\sum_{j=1}^{\mu(A)^{-1}}  (\mu(A)^{-1}-j)\, \mu(A\cap\sigma^{-j} (A)) \,.
\end{equation}

Let us first consider the $\phi$-mixing case. 
For $j\ge n$, mixing gives $ \mu(A\cap\sigma^{-j} (A))\le \mu(A)^2+\mu(A)\phi(j-n+1)$. 
Thus, 
\begin{align} \label{chocho}
\sum_{j=n}^{\mu(A)^{-1}}  (\mu(A)^{-1}-j)\, \mu(A\cap\sigma^{-j} (A))\le \frac{1}{2} +\sum_{\ell=0}^{\mu(A)^{-1}-n } \phi(\ell+1) 
\end{align}
where we used $\mu(A)^{-1}-j\leq\mu(A)^{-1}$ to get the last term.

For $1\leq j\le n-1$, as before $ A^{(j)} \subset    A^{(\lceil j/2 \rceil)}  $, thus 
\begin{align*}
\mu\left(A\cap\sigma^{-j} \left( A\right)\right) 
& = \mu\left(A\cap \sigma^{-n}\left(A^{(j)}\right)\right)\\
& \leq \mu\left(A\cap\sigma^{-n-\lfloor j/2 \rfloor}\left( A^{(\lceil j/2 \rceil)}\right)\right) \\
& \le \mu( A) \left( \mu\left(A^{\left(\lceil j/2 \rceil\right)}\right)  + \phi( \lfloor j/2+1 \rfloor) \right)\\
&\le \mu(A)\left(Ce^{-c\lceil j/2 \rceil}+\phi( \lfloor j/2 +1\rfloor)\right)\,.
\end{align*}
Therefore, 
\begin{align}
\sum_{j=1}^{n-1} (\mu(A)^{-1}-j)\, \mu(A\cap\sigma^{-j} (A))\leq \sum_{j=1}^{n-1} \left( Ce^{-c\lceil j/2 \rceil}+\phi( \lfloor j/2 +1\rfloor) \right).
\label{chachou}
\end{align}
Therefore, by \eqref{chocho} and \eqref{chachou}, the summability of $\phi$ concludes the proof for the $\phi$-mixing case.

If $\mu$ is $\psi$-mixing, we separate the sum in (\ref{geral}) in three parts. First, recall the definition of $g_0$ in Section \ref{sec:correct}. For $1\leq j\leq g_0$, we bound the sum as follows
\[
\sum_{j=1}^{g_0} (\mu(A)^{-1}-j)\, \mu(A\cap\sigma^{-j} (A)) \leq \sum_{j=1}^{g_0}\mu(A)^{-1}\mu(A)=g_0.
\]
For $g_0+1\leq j \leq g_0+n-1$, we have by $\psi$-mixing
\[
(\mu(A)^{-1}-j)\, \mu(A\cap\sigma^{-j} (A))\leq \mu(A)^{-1}\mu\left(A\cap\sigma^{-n-g_0} \left(A^{(j-g_0)}\right)\right)\le M\, \mu(A)^{-1}\mu(A)\mu\left(A^{(\ell)}\right)
\]
where we denoted $\ell=j-g_0$. Thus
\[
\sum_{j=g_0+1}^{g_0+n-1}  (\mu(A)^{-1}-j)\, \mu(A\cap\sigma^{-j} (A))\le  M\sum_{\ell=1}^{n-1}Ce^{-c\ell}.
\]
Finally, applying $\psi$-mixing again, 
\[
\sum_{j=g_0+n}^{\mu(A)^{-1}}  (\mu(A)^{-1}-j)\, \mu(A\cap\sigma^{-j} (A))\le M\sum_{j=n+g_0}^{\mu(A)^{-1}}  \mu(A)^{-1}\,\mu(A)^2 \leq M,
\]
concluding the proof of the $\psi$-mixing case.


%
%

\end{proof}

\bibliography{potentialbib}
\bibliographystyle{alpha}


\end{document}